\documentclass[12pt]{article}

\usepackage{amsmath}
\usepackage{amssymb}
\usepackage{amsthm}

\newtheorem{theorem}{Theorem}
\newtheorem{lemma}{Lemma}
\newtheorem{corollary}{Corollary}
\newtheorem{proposition}{Proposition}

\newtheorem*{example*}{Example}
\newtheorem*{remark*}{Remark}

\def\Z{\mathbb{Z}}
\def\Q{\mathbb{Q}}
\def\OK{\mathcal{O}_K}
\def\eps{\varepsilon}
\def\I{\mathcal{I}}

\DeclareMathOperator{\sgn}{sgn}

\def\ndiv{\not |\;}

\title{Statistics of genus numbers of cubic fields}
\author{Kevin J. McGown and Amanda Tucker}

\begin{document}

\maketitle

\begin{abstract}
We prove that approximately $96.23\%$ of cubic fields, ordered by discriminant, have genus number one,
and we compute the exact proportion of cubic fields with a given genus number.
We also compute the average genus number.
Finally, we show that a positive proportion of totally real cubic fields with genus number one fail to be norm-Euclidean.
\end{abstract}

\section{Introduction}

The genus theory of algebraic number fields can be traced back to
Gauss's celebrated work on binary quadratic forms, and has its
roots in earlier work of Euler, Lagrange, and others.
It was Hasse who first defined the genus field of a quadratic extension when
he reproved a classical theorem of Gauss using class field theory~\cite{hasse},
and Leopoldt later defined the genus field of any absolutely abelian extension~\cite{leopoldt}.
The definition of the genus field of a general number field, which we give forthwith, is due to Fr\"ohlich~\cite{frohlich}.

The genus field of a number field $K$ is defined to be the maximal extension $K^*$ of $K$ that is unramified at all finite primes and is a compositum of the form $Kk^*$ where $k^*$ is absolutely abelian.  The genus number of $K$ is defined as $g_K=[K^*:K]$.
It follows right away that $g_K$ divides $h^+_K$, the narrow class number of $K$.
Since the class number $h_K$ and the narrow class number $h_K^+$ differ by a power of $2$ and the genus number of a cubic field is a power of $3$ (see Theorem~\ref{T:frohlich}), it follows that $g_K$ divides $h_K$ when $K$ is cubic.
%The class number $h_K$ and the narrow class number $h^+_K$ differ by a power of~$2$, so because the genus number of a cubic field is a power of $3$, we get that $g_K$ divides $h_K$ for $K$ cubic.

%Indeed, if we consider the Hilbert class field $H$, which is the maximal abelian extension of $K$ unramified at all primes (finite and infinite), class field theory tells us that the Galois group $G(H/K)$ is isomorphic to the class group $\Cl_K$ and consequently $h_K:=\#\Cl_K=[H:K]$.

The class number is among the most important invariants associated to a number field,
but it is very difficult to study.
Conjecturally, its behavior at the ``good'' primes is
governed by the (modified) heuristics of Cohen--Lenstra--Martinet (see~\cite{cohen.lenstra, cohen.martinet}).
By contrast, the genus number (whose support is at ``bad'' primes) does not behave ``randomly'' and  is therefore more amenable to study. 
It is very natural to ask about the density of genus number one fields
among all number fields of a fixed degree and signature.
In the present investigation, we will discuss the situation for cubic fields,
as this is the simplest situation where this question has not been previously addressed.

Let $K$ be a cubic field.
If $K$ is cyclic, then $g_K=3^{e-1}$, where $e$ is the number of odd prime factors of the discriminant $\Delta$ of $K$;
it follows right away that $0\%$ of cyclic cubic fields have genus number one,
and that the average genus number in this setting is infinite.
These same statistical questions become more subtle when one does not impose the restriction
that $K$ is Galois.  In fact, since $0\%$ of cubic fields are cyclic, the aforementioned facts have
little bearing on the answers when one considers the collection of all cubic fields.
In this paper, we show that roughly $96.23\%$ of cubic fields have $g_K=1$.
In addition, we prove that the average genus number is roughly $1.0785$.

Let $\mathcal{F}$ denote the collection of all cubic fields $K$
with $g_K=1$, and write $\mathcal{F}^+$, $\mathcal{F}^-$ to denote the subsets of $\mathcal{F}$
consisting of fields with positive and negative discriminants, respectively.
Set $N^\pm(X)=\#\{K\in\mathcal{F}^\pm: |\Delta|\leq X\}$ and define constants
$n^+=6$ and $n^-=2$.  (Note that we will always count cubic fields up to isomorphism.)

In Section \ref{genus1}, we prove our main result:
\begin{theorem}\label{T:1}
\begin{align*}
N^\pm(X)
&=
\frac{29}{54n^\pm\zeta(2)}\prod_{\substack{p\equiv 2\pmod{3}}} \left(1+\frac{1}{p(p+1)}\right)X+O\left(X^{16/17+\eps}\right)
\end{align*}
\end{theorem}

\begin{corollary}\label{C:1}
The proportion of cubic fields with genus number one
(of positive or negative discriminant)
equals
\[
\frac{29\,\zeta(3)}{27\,\zeta(2)}\prod_{\substack{p\equiv 2\pmod{3}}} \left(1+\frac{1}{p(p+1)}\right)
\,.
\]
Consequently, roughly $96.23009\%$ of totally real cubic fields and $96.23009\%$ of complex cubic fields have genus number one.
\end{corollary}

In Section \ref{ave}, we prove the following result regarding the average genus number of a cubic field:

\begin{theorem}\label{T:2}
The average genus number of a cubic field
(in the positive or negative discriminant case)
is given by

\begin{align*}
  &
  \lim_{X\to\infty}\frac{\displaystyle \sum_{0<\pm\Delta\leq X}g_K}{\displaystyle \sum_{0<\pm\Delta\leq X}1}
  \\[2ex]
  &\qquad
  =
    \frac{119\zeta(3)}{108\zeta(2)}
  \prod_{p\equiv 1\pmod{3}} \left(1+\frac{3}{p(p+1)}\right)\prod_{p\equiv 2\pmod{3}} \left(1+\frac{1}{p(p+1)}\right)
  \\[1ex]
  &\qquad
  \approx
1.078541 
\end{align*}
The above sums are taken over all cubic fields $K$
where the discriminant $\Delta$ falls in the specified range.
\end{theorem}

In Section~\ref{dist}, we give the exact proportion of cubic fields with a given genus number. 
\begin{theorem} \label{T:3}
A positive proportion of cubic fields (of positive or negative discriminant)
have $g_K=m$ iff $m$ is a power of $3$, and
the exact proportion with $g_K=3^k$ is given by
\[
  \frac{\zeta(3)}{\zeta(2)}
  \left[
  \frac{29}{27}\sum_{f\in T_k} \prod_{p|f}\frac{1}{p(p+1)}
  +
    \frac{1}{108}\sum_{f\in T_{k-1}} \prod_{p|f}\frac{1}{p(p+1)}
    \right]
  \,,
\]
where $T_k$ denotes the collection of squarefree integers coprime to $3$ having exactly $k$ prime factors $p$ that  satisfy
$p\equiv 1\pmod{3}$.\footnote{If we adopt the convention that $T_{-1}=\emptyset$, then the formula holds for $k=0$ as well.}
The approximate proportions are given in the following table:
\begin{center}
\begin{tabular}{|c|c|c|c|c|c|}
\hline
$k$ & $0$ & $1$ & $2$ & $3$  \\
\hline
proportion & $96.23\%$ & $3.72\%$ & $0.05\%$ & really small! \\
\hline
\end{tabular}
\end{center}
\end{theorem}

Our initial interest in this question stemmed from the study of norm-Euclidean cubic fields.
There are only finitely many norm-Euclidean cubic fields with negative discriminant,
but it may very well be the case that there are infinitely many with positive discriminant
(see~\cite{davenport.euclidean, heilbronn.cubic}).
A norm-Euclidean field is necessarily class number one and, hence, genus number one;
thus, if $g_K \neq 1$, we can trivially conclude that $K$ is not norm-Euclidean.
Consequently, it is of greater interest to study fields that fail to be norm-Euclidean for reasons other than genus theory.
In Section \ref{NE}, we prove the following result:

\begin{theorem}\label{T:euclidean.genus}
A positive proportion of totally real cubic fields with genus number one fail to be norm-Euclidean.
\end{theorem}

Our starting point is a theorem of Fr\"ohlich which gives an explicit
description of $g_K$ when $K$ is cubic.
The main tool we employ is a powerful theorem,
proved independently by Taniguchi--Thorne
and Bhargava--Shankar--Tsimerman,
that allows one to compute the density of  cubic discriminants 
satisfying specified local conditions with a very precise error term.
This is a generalization of a classical theorem
of Davenport--Heilbronn, who were the first to accomplish the counting of cubic fields.

%%%%%%%% Preliminaries %%%%%%%%%%%

\section{Preliminaries}
Let $K$ be a cubic field. Then the discriminant $\Delta$ takes one of the three forms $df^2$, $9df^2$, $81df^2$, where $d$ is a fundamental discriminant and $f$ is a squarefree positive integer coprime to $3$.
A prime $p \neq 3$ is totally ramified in $K$ if and only if $p$ divides $f$,
and $3$ is totally ramified in $K$ if and only if $\Delta$ takes one of the forms $9df^2$, $81df^2$~\cite{cohen}.  The following theorem of Fr\"ohlich gives an explicit expression for the genus number~\cite{frohlich};
see also~\cite{ishida}.

%\begin{theorem}[Fr\"ohlich]\label{T:frohlich}
%Let $e$ denote the number of primes $p\neq 2,3$ such that $p$ is totally ramified in $K$ and $p\equiv 1\pmod{3}$,
%and add $+1$ to $e$ if $\Delta=81df^2$ with $df^2\equiv 1\pmod{3}$.
%Then we have:
%\[
%   g_K=\begin{cases}
%   3^{e-1} & \text{if $K$ is cyclic}\\
%   3^e & \text{if $K$ is not cyclic.}
%   \end{cases}
%\]
%%If $K$ is not cyclic, then $g_K=3^{e}$.
%\end{theorem}
%Here's maybe a better way to say Frohlich's theorem.  Is this still correct when $K$ is cyclic?
\begin{theorem}[Fr\"ohlich]\label{T:frohlich}
Let $e$ denote the number of odd primes $p$ such that $p$ is totally ramified in $K$ and $(d/p)=1$,
where $(d/p)$ is the usual Legendre symbol.
Then we have:
\[
   g_K=\begin{cases}
   3^{e-1} & \text{if $K$ is cyclic}\\
   3^e & \text{if $K$ is not cyclic.}
   \end{cases}
\]
%If $K$ is not cyclic, then $g_K=3^{e}$.
\end{theorem}
%The cyclic case of the previous theorem was already known to Leopoldt~\cite{???}.
Note that since the number of cyclic cubic fields with discriminant less than or equal to $X$ is $O(X^{1/2})$,
for our purposes we may neglect these fields~\cite{hasse3, cohn}.

Our main tool is the following theorem, which
is a strengthening of the classical Davenport--Heilbronn Theorem
(see~\cite{davenport.heilbronn, taniguchi.thorne, bhargava.shankar.tsimerman}).
For what follows, set $m^+=1$ and $m^-=\sqrt{3}$.
\begin{theorem}[Taniguchi--Thorne, Bhargava--Shankar--Tsimerman]\label{T:TTBST}
The number of cubic fields satisfying $0<\pm\Delta\leq X$ equals
\[
  \frac{1}{2n^\pm\zeta(3)}X+\frac{4m^\pm\zeta(1/3)}{5\Gamma(2/3)^3\zeta(5/3)}X^{5/6}+O(X^{7/9+\eps})
  \,.
\]
\end{theorem}
We note in passing that the secondary term was conjectured
by Roberts~\cite{roberts} following computations carried out by Belabas~\cite{belabas1},
and that the existence of a power saving error term was first proved
by Belabas--Bhargava--Pomerance in~\cite{belabas.bhargava.pomerance}; this whole story is summarized
very nicely in~\cite{belabas.bhargava.pomerance}.

In fact, both papers~\cite{taniguchi.thorne, bhargava.shankar.tsimerman}
give a stronger version of Theorem~\ref{T:TTBST} 
(which we will require) allowing one to specify local conditions.
If local conditions are imposed at finitely many primes $p$,
then the main term and the secondary term are multiplied by an additional
factor for each $p$; moreover, in this case,
the implicit constant in the $O$-term now depends upon the set of local conditions.
For our application we need to specify
infinitely many local conditions \emph{and} we require an explicit dependence on these local conditions.
In principle, either Theorem~7 of~\cite{bhargava.shankar.tsimerman} or Theorem~1.3 of~\cite{taniguchi.thorne} will suffice,
but the situation is such that neither accomplishes our aim ``out-of-the-box'' with no additional work.
We have chosen to use Theorem~1.3 of~\cite{taniguchi.thorne} as our main work horse.
At the appropriate juncture in our proofs, please see \S6 of~\cite{taniguchi.thorne} 
and \S4 of~\cite{bhargava.shankar.tsimerman}
for information regarding local density calculations.

Finally, we mention that there is a forthcoming paper~\cite{bhargava.taniguchi.thorne}
that improves the error term in Theorem~\ref{T:TTBST} to $O(X^{2/3+\eps})$.
Substituting this result into our arguments would result in an improvement of our error terms.

%%%%% COUNTING GK = 1 CUBIC FIELD %%%%%%%%%%%%%%%%%

\section{Counting genus number one cubic fields}\label{genus1}

Recall that we write the discriminant of a cubic field $K$ as $\Delta = df^2$, $9df^2$, or $81df^2$ with $f$ coprime to $3$. Let $\mathcal{G}^\pm$ denote the collection of all cubic fields with $\sgn(\Delta)=\pm 1$, and let $\mathcal{F}^\pm \subseteq \mathcal{G}^\pm$ denote the collection of all such cubic fields $K$ with $g_K=1$.  Let $N^\pm(X)=\#\{K\in\mathcal{F}^\pm: |\Delta|\leq X\}$. Define
\begin{align*}
\mathcal{F}^\pm_1&=\{K\in\mathcal{G}^\pm\mid p\equiv 2\pmod{3} \text{ for all $p$ dividing $f$}\}\\
\mathcal{F}^\pm_2&=\{K\in\mathcal{F}^\pm_1\mid \text{ $3$ is totally ramified and } d\equiv 1\pmod{3}\}.
\end{align*}
By Theorem~\ref{T:frohlich}, we have $\mathcal{F}^\pm=\mathcal{F}^\pm_1\setminus\mathcal{F}^\pm_2$,
and therefore $N^\pm(X)=N^\pm_1(X)-N^\pm_2(X)$ where
$N^\pm_i(X)=\#\{K\in\mathcal{F}^\pm_i : |\Delta|\leq X\}$.
Indeed, when $p\neq 2,3$ is totally ramified, the condition $(d/p)=1$ is equivalent to $p\equiv 1\pmod{3}$
(see~\cite{ishida}, Example 6.10),
and, of course, $(d/3)=1$ is equivalent to $d\equiv 1\pmod{3}$.
In what follows, we will establish asymptotic formulas for $N^\pm_1(X)$, $N^\pm_2(X)$ individually
and then obtain the desired result by subtraction.

For each square free $f$ coprime to $3$ we write $N^\pm(f;X)$
to denote the number of fields in $\mathcal{G}^\pm$ with $|\Delta|\leq X$ that are totally ramified at the primes dividing $f$ and at no other primes, except possibly $3$. 

\begin{proposition}\label{incl-ex}
\[
N^\pm(f;X)
=
   \frac{13}{24n^\pm\zeta(2)}\prod_{p|f}\frac{1}{p(p+1)} X
   +
O(f^{-1}X^{16/17+\eps})
\,.
\]
\end{proposition}

\begin{proof}
Write $M^\pm(f;X)$ to denote the number of cubic fields of positive (or negative) discriminant
with $|\Delta|\leq X$ where all the primes dividing $f$ are totally ramified (and no other restrictions).
Observe that this constitutes only finitely many local conditions.
Notice that
\[
  N^\pm(f;X)=\sum_{(r,3f)=1}\mu(r)M^\pm(rf;X)
  \,.
\]
We will split this sum
as $\sum_r=\sum_{r\leq Y}+\sum_{r>Y}$ where $Y$
is some parameter to be specified.
In all sums over $r$ we will only consider values where $(r,3f)=1$.
For $r\leq Y$, we will use Theorem~1.3 of~\cite{taniguchi.thorne} to obtain
\[
  M^\pm(f;X)= 
  c_f^\pm X + 
  d_f^\pm X^{5/6} + O(f^{16/9}X^{7/9+\epsilon})
  \,,
\]
where the constant in the main term is

\begin{align*}
  c_f^\pm
  &=
  \frac{1}{2n^\pm\zeta(3)}
  \prod_{p|f}
  \frac{1}{p^2+p+1}
  \,,
\end{align*}
and the constant in the secondary term is
\begin{align*}
  d_f^\pm=
  \frac{4m^\pm\zeta(1/3)}{5\Gamma(2/3)^3\zeta(5/3)}
  \prod_{p|f}\frac{p^{2/3}-1}{(p^{5/3}-1)(p-1)}
  =
  O(f^{-2})
  \,.
\end{align*}
On the other hand, when $r>Y$ we will use the estimate 
\[
  M^\pm(f;X)=O(f^{-2+\eps}X)
  \,,
\]
which follows immediately from Lemma~3.4 of~\cite{taniguchi.thorne}, in light of the fact  that $6^{\omega(f)}=O(f^\eps)$.
Consequently,
\begin{align*}
%&
N^\pm(f;X)
%\\
&=
\sum_{r\leq Y}
\left(
\mu(r)c_{rf}X+
%\sum_{r\leq Y}
O((rf)^{-2}X^{5/6})
+
%\sum_{r\leq Y}
O((rf)^{16/9}X^{7/9+\eps})
\right)
+
\sum_{r>Y}O((rf)^{-2+\eps}X)
\\
&=
X\sum_{r\leq Y}\mu(r)c_{rf}
+
O(f^{-2}X^{5/6})+O(Y^{25/9}f^{16/9}X^{7/9 +\eps})+O(Y^{-1+\eps}f^{-2+\eps}X)
\,.
\end{align*}
We compute the constant in the main term:
\begin{align*}
  \sum_{r\leq Y}\mu(r)c_{rf}
  &=
  \sum_{r}\mu(r)c_{rf} + O\left(  \sum_{r> Y}c_{rf}\right)
  \\
  &=
  \frac{1}{2\zeta(3)n^\pm}\prod_{p|f}\frac{1}{p^2+p+1}\sum_{(r,3f)=1}\mu(r)\prod_{p|r}\frac{1}{p^2+p+1}
  +
O\left(  \sum_{r> Y}(rf)^{-2+\eps}\right)
  \\
  &=
  \frac{1}{2\zeta(3)n^\pm}\prod_{p|f}\frac{1}{p^2+p+1}\prod_{p \not  \hspace{2pt} | \, 3f}\left(1-\frac{1}{p^2+p+1}\right)
  +
 O(Y^{-1+\eps}f^{-2+\eps})
 \\
 &=
   \frac{13}{24\zeta(3)n^\pm}\prod_{p|f}\frac{1}{p(p+1)}\prod_{p}\left(1-\frac{1}{p^2+p+1}\right)
  +
 O(Y^{-1+\eps}f^{-2+\eps})
 \\
 &=
   \frac{13}{24\zeta(2)n^\pm}\prod_{p|f}\frac{1}{p(p+1)}
  +
 O(Y^{-1+\eps}f^{-2+\eps})
\end{align*}
Setting $Y=f^{-1}X^{1/17}$ and putting this all together yields the result.
%Therefore we have
%\begin{align*}
%  N(f;X)
%  &=
%   \frac{13}{24\zeta(2)n^\pm}\prod_{p|f}\frac{1}{p(p+1)} \cdot X
%  +   O(f^{-2}X^{5/6})
%  +
%  O(Y^{25/9}f^{16/9}X^{7/9+\eps})+O(Y^{-1+\eps}f^{-2+\eps}X)
%  \,.
%\end{align*}
%
\end{proof}

Let $T$ denote the collection of all squarefree $f$ with the property that $p\equiv 2\pmod{3}$ for all primes $p$ dividing $f$.
We have 
\begin{align}
 \nonumber
  N^\pm_1(X) &= \sum_{f\in T}N^\pm(f;X)
  \\
  \nonumber
  &= \sum_{\substack{f\in T\\f\leq X^{1/2}}  }
  \left(     \frac{13}{24\zeta(2)n^\pm}\prod_{p|f}\frac{1}{p(p+1)}\cdot X + O(f^{-1}X^{16/17+\eps}) \right)
 \\
 \nonumber
&=
\frac{13}{24\zeta(2)n^\pm}\cdot X  \sum_{f\in T}\prod_{p|f}\frac{1}{p(p+1)}
+
X\sum_{f>X^{1/2}} O(f^{-2})
+
\sum_{f\leq X^{1/2}}O(f^{-1}X^{16/17+\eps})
\\
\label{E:manip}
&=
 \frac{13}{24\zeta(2)n^\pm}\prod_{p\in T}\left(1+\frac{1}{p(p+1)}\right)\cdot X + O(X^{16/17+\eps}).
 \end{align}
 
% \begin{bf} INCLUDE THIS???\end{bf}
% The third equality follows because the error in the first sum from dropping the restriction $f\leq X^{1/2}$ is $O(X^{1/2})$ and the second sum is only made larger if we sum over all $f$ instead of $f \in T$.
 
This establishes the desired formula for $N^\pm_1(X)$.  In order to deal with $N^\pm_2(X)$ we require
a slight modification of the quantity $N(f;X)$.
For each squarefree $f$ coprime to $3$ we write $N'(f;X)$ to denote the number of such fields that are totally ramified at $3$ and the primes dividing $f$ but no other primes \emph{and}
that also satisfy the extra condition $d\equiv 1\pmod{3}$.
(For the remainder of this section we have dropped $\pm$ from most of the notation.)

 \begin{proposition}\label{incl-ex2}
 \[
N'(f;X)
=
   \frac{1}{216\zeta(2)n^\pm}\prod_{p|f}\frac{1}{p(p+1)}  X
   +
O(f^{-1}X^{16/17+\eps})
\,.
\]
 \end{proposition}
 
\begin{proof}
As before, we can write
\[
  N'(f;X)=\sum_{(r,3f)=1}\mu(r)M'(rf;X)
\]
and apply Theorem~1.3 of~\cite{taniguchi.thorne} to obtain
\[
  M'(f;X)= 
  c'_f X
 + O(f^{-2}X^{5/6}) + O(f^{16/9}X^{7/9+\eps})
  \,.
\]
The calculation of $c'_f$ is identical to that of $c_f$ except for the local factor at the prime $3$.
We need to impose the additional conditions that $3$ is totally ramified and $d\equiv 1\pmod{3}$.
These conditions are definable in terms of congruence conditions on the coefficients of the corresponding cubic form;
indeed, this is equivalent to saying $\Delta\equiv 3^4\pmod{3^5}$.

At this juncture, computing the local densities on the ``forms side'' via the Delone--Fadeev correspondence
is more convenient; please see \S4 of~\cite{bhargava.shankar.tsimerman} for details regarding
this type of local computation.
Let $S$ denote the collection of integral binary cubic forms having a triple root modulo $3$,
satisfying the maximality condition at $3$,
and satisfying our congruence condition on the discriminant.
Let $\mu_3(S)$ denote the $3$-adic density of the $p$-adic closure of $S$ in $\Z_3^4$.
(Here the additive measure $\mu_3$ is normalized so that $\mu_3(\Z_3^4)=1$.)
In this notation, we have
\[
  c'_f=\mu_3(S)(1-3^{-2})^{-1}(1-3^{-3})^{-1}c_f
  \,.
\]
All that remains is to compute $\mu_3(S)$.
Since maximality is a condition modulo~$3^2$, we may do all our calculation modulo~$3^2$.
%The calculation of $c'_f$ is identical to that of $c_f$ except that $\mu_3(S)$ in substituted in place of
%$(1-3^{-2})(1-3^{-3})$.

There are $8$ forms over $\Z/3\Z$ with a triple root
and $2/3$ of the lifts of these forms to $\Z/3^2\Z$ are maximal at $3$.
For each of these $432$ forms, we compute the discriminant via the standard formula and then check whether the fundamental part of the discriminant
satisfies $d\equiv 1\pmod{3}$.  Precisely $48$ of these fit the bill, in other words, $1/9$ of the forms under consideration.
It follows that the $3$-adic density is
$\mu_3(S)=(8/3^4)(2/3)(1/9)=16/2187$.

Observe that
\[
  \frac{16/2187}{(1-3^{-2})(1-3^{-3})}=1/117
  \,,
\]
which leads to
\[
  c_f'
  =
    \frac{1}{234\zeta(3)n^\pm}
  \prod_{p|f}
  \frac{1}{p^2+p+1}
  \,.
\]
(The extra factor can also be computed as $(1/9)(3^2+3+1)^{-1}=1/117$.)
The rest of the proof proceeds exactly as in the proof of Proposition~\ref{incl-ex}.
 \end{proof}

%\subsection{Computing the $3$-adic density for $N^\pm_2(X)$}
%\subsection{The final calculation}
Applying the previous proposition and following the same procedure
we used to obtain our formula for $N_1(X)$
yields
\[
  N_2(X)=
   X
     \frac{1}{216n^\pm \zeta(2)}
   \prod_{p\in T}\left(1+  \frac{1}{p(p+1)}\right) + O\left(X^{16/17+\eps}\right).
\]
Finally, subtracting, we have
\[
  N(X)
  = 
   X
     \frac{29}{54n^\pm\zeta(2)}
   \prod_{\substack{p\equiv 2\pmod{3}}}\left(1+  \frac{1}{p(p+1)}\right) + O\left(X^{16/17+\eps}\right).
\]
This proves Theorem~\ref{T:1}, and Corollary~\ref{C:1} follows.

%%%%%%%%%%%%%%%%%%%%%%%%%%%%%%%%%%%%%%%%%%%%%%%%%%%%%%%%%%%%

\section{The average genus number}\label{ave}
%\subsection{The cyclic case}
First, we verify that the cyclic fields do not contribute to the average.
When $K$ is cyclic, we have $\Delta\in\{f^2, (3f)^2, (9f)^2\}$ and,
by Theorem~\ref{T:frohlich}, $ g_K =   3^{e-1} $;
in this case, $e$ is the number of odd primes that are (totally) ramified in $K$.
 Therefore, we have
\[
  \sum_{\substack{0<\pm \Delta\leq X\\ K\text{cyclic}}}
  g_K
  \leq
  \sum_{k=1}^\infty \pi_k(\sqrt{X})3^{k-1}
  \,,
\]
where $\pi_k(y)$ denotes the number of positive integers $\leq y$ with exactly $k$ prime factors.
Setting $y=\sqrt{X}$ and using the fact (see~\cite{landau}) that 
\[
  \pi_k(y)\sim\frac{y}{\log y}\frac{(\log\log y)^{k-1}}{(k-1)!}
\]
we obtain
\[
  \sum_{\substack{0<\pm \Delta\leq X\\ K\text{cyclic}}}
  g_K
  =
  O\left(
  \frac{y}{\log y}
  \sum_{k=1}^\infty
  \frac{(3\log\log y)^{k-1}}{(k-1)!}
  \right)
  =
  O\left(
  X^{1/2}(\log X)^2
  \right)
  \,.
\]
The fact that the above expression is $o(X)$ tells us that
cyclic fields do not contribute to the average genus number.

%\subsection{The non-Galois case}
We now turn to the main part of the proof.
%We now compute the average genus number, safely neglecting the cyclic fields.
Define $\psi(n)$ to be the number of primes $p$ dividing $n$ satisfying $p\equiv 1\pmod{3}$.
As we are ignoring cyclic fields, everything that follows holds up to an error of $O(X^{1/2+\eps})$.
We have
\begin{align*}
  \sum_{0<\pm \Delta\leq X}g_K
  &=
    \sum_{0<\pm \Delta\leq X}  3^{\psi(f)}
  -\sum'_{\substack{0<\pm \Delta\leq X}}  3^{\psi(f)}
  +\sum'_{\substack{0<\pm \Delta\leq X}}  3^{{\psi(f)}+1}
  \\
    &=
  \sum_{0<\pm \Delta\leq X}  3^{\psi(f)}
  +2\sum'_{\substack{0<\pm \Delta\leq X}}  3^{\psi(f)}
  \\
  &=
  \sum_{(f,3)=1}^\flat 3^{\psi(f)} N(f;X)
   +2
  \sum_{(f,3)=1}^\flat 3^{\psi(f)} N'(f;X),
\end{align*}
where $\displaystyle \sum'$ denotes only summing over those fields where $3$ is totally ramified
with $d\equiv 1\pmod{3}$
and the $\displaystyle \sum^\flat$ denotes summing over squarefree $f$.

Applying Proposition~\ref{incl-ex} to compute the first sum above, we obtain:
\begin{align*}
  \sum_{(f,3)=1}^\flat 3^{\psi(f)} N(f;X)
   &=
  \frac{13}{24\zeta(2)n^\pm} X 
    \sum_{\substack{(f,3)=1\\f\leq X^{1/2}}}^\flat 3^{\psi(f)} \prod_{p|f}\frac{1}{p(p+1)}
    +
        \sum_{\substack{(f,3)=1\\f\leq X^{1/2}}}^\flat 3^{\psi(f)}
        O(f^{-1}X^{16/17+\eps})\,
\end{align*}
After performing manipulations similar to (\ref{E:manip}) this yields
\[
          \frac{13}{24n^\pm\zeta(2)}X
  \prod_{p\neq 3} \left(1+\frac{3^{\psi(p)}}{p(p+1)}\right)
  +
  \sum_{f>X^{1/2}}3^{\psi(f)}O(f^{-2}X)+
   \sum_{f\leq X^{1/2}} 3^{\psi(f)}O(f^{-1}X^{16/17+\eps})\,.
\]
The main term is
\begin{align*}
%&
%  \frac{13}{24n^\pm\zeta(2)}X
%  \prod_{p\neq 3} \left(1+\frac{3^{\psi(p)}}{p(p+1)}\right)
%  \\
%  &=
  \frac{13}{24n^\pm\zeta(2)}
  X\prod_{p\equiv 1\pmod{3}} \left(1+\frac{3}{p(p+1)}\right)\prod_{p\equiv 2\pmod{3}} \left(1+\frac{1}{p(p+1)}\right).
\end{align*}
Because $3^{\psi(f)}=O(f^\eps)$, the error term is
\begin{align*}
  \sum_{f>X^{1/2}}O(f^{-2+\eps}X)+
   \sum_{f\leq X^{1/2}} O(f^{-1+\eps}X^{16/17+\eps})\,
%\sum_{\substack{(f,3)=1\\f\leq X^{1/2}}}O(f^{-1+\eps}X^{16/17+\eps})
=
O(X^{16/17+\eps})
\,.
\end{align*}

In exactly the same manner, we apply Proposition~\ref{incl-ex2} to
compute the second term
\[
  \sum_{(f,3)=1}^\flat 3^{\psi(f)} N'(f;X) 
  \,.
\]
This simply results in multiplying the first outcome by a factor of $1/117$.  Thus the whole sum is the first term multiplied by $1+2/117=119/117$.  Hence, we obtain 
\[
 \frac{119}{216n^\pm\zeta(2)}
  X\prod_{p\equiv 1\pmod{3}} \left(1+\frac{3}{p(p+1)}\right)\prod_{p\equiv 2\pmod{3}} \left(1+\frac{1}{p(p+1)}\right)
 +O(X^{16/17+\eps})
 \,.
\]
Dividing the above by $1/(2n^\pm\zeta(3))$ yields the desired expression,
thereby proving Theorem~\ref{T:2}.
%So the average genus number is
%\[
%  \frac{119\zeta(3)}{108\zeta(2)}
%  \prod_{p\equiv 1\pmod{3}} \left(1+\frac{3}{p(p+1)}\right)\prod_{\substack{p\equiv 2\pmod{3}}} \left(1+\frac{1}{p(p+1)}\right) \approx
%    1.078541,
%\]
%which proves Theorem \ref{T:2}.

\section{Counting cubic fields with given genus number} \label{dist}

 As before, $\mathcal{G}^\pm$ will denote the collection of all cubic fields with $\sgn(\Delta)=\pm 1$.
 We now let $\mathcal{F}^\pm \subseteq \mathcal{G}^\pm$ denote the collection of all cubic fields $K$ with $g_K=3^k$. As before, define $N^\pm(X)=\#\{K\in\mathcal{F}^\pm: |\Delta|\leq X\}$. Let $T_k$ denote the collection of squarefree integers $n$ coprime to $3$ with $\psi(n) = k$ (i.e., having exactly $k$ prime factors $p$ satisfying $p\equiv 1\pmod{3}$).
\begin{align*}
\mathcal{F}^\pm_1&=\{K\in\mathcal{G}^\pm\mid f \in T_k\}\\
\mathcal{F}^\pm_2&=\{K\in\mathcal{G}^\pm\mid f\in T_k, \text{ $3$ is totally ramified, and } d\equiv 1\pmod{3}\}\\
\mathcal{F}^\pm_3&=\{K\in\mathcal{G}^\pm\mid f\in T_{k-1}, \text{ $3$ is totally ramified, and } d\equiv 1\pmod{3}\}.
\end{align*}

By Theorem~\ref{T:frohlich}, we have $\mathcal{F}^\pm=(\mathcal{F}^\pm_1\setminus\mathcal{F}^\pm_2)\cup\mathcal{F}^\pm_3$,
and therefore $N^\pm(X)=N^\pm_1(X) - N^\pm_2(X)+N^\pm_3(X)$ where
$N^\pm_i(X)=\#\{K\in\mathcal{F}^\pm_i : |\Delta|\leq X\}$.
Now we can proceed exactly as in Section~\ref{genus1} to find
\[
N^\pm_1(X) =  \frac{13}{24\zeta(2)n^\pm}X \sum_{f\in T_k} \prod_{p|f}\frac{1}{p(p+1)}
%+ \sum_{f\leq X^{1/2}}O(f^{-1}X^{16/17+\eps})
+O(X^{16/17+\eps})
\]
and we multiply by $1/117$  to obtain
\[
N^\pm_2(X) =  \frac{1}{216\zeta(2)n^\pm}X \sum_{f\in T_k} \prod_{p|f}\frac{1}{p(p+1)}
+ O(X^{16/17+\eps})
\,.
\]
Similarly, we obtain
\[
N^\pm_3(X) =  \frac{1}{216\zeta(2)n^\pm}X \sum_{f\in T_{k-1}} \prod_{p|f}\frac{1}{p(p+1)}
+ O(X^{16/17+\eps})
\,,
\]
and this makes the desired proportion equal to
\[
  \frac{\zeta(3)}{\zeta(2)}
  \left[
  \frac{29}{27}\sum_{f\in T_k} \prod_{p|f}\frac{1}{p(p+1)}
  +
    \frac{1}{108}\sum_{f\in T_{k-1}} \prod_{p|f}\frac{1}{p(p+1)}
    \right]
  \, .
\]
We include here a table of approximations to the first few percentages.
\begin{center}
\begin{tabular}{|c|c|c|c|c|c|}
\hline
$k$ & $0$ & $1$ & $2$ & $3$  \\
\hline
proportion & $96.23\%$ & $3.72\%$ & $0.05\%$ & really small! \\
\hline
\end{tabular}
\end{center}
This concludes the proof of Theorem~\ref{T:3}.

We note that the formulas for $N^\pm_1$, $N^\pm_2$, $N^\pm_3$ just derived can
essentially be used to establish Theorems~\ref{T:1},~\ref{T:2}, and~\ref{T:3}
but we have chosen to structure the paper in this manner for clarity of exposition.

\section{Norm-Euclidean cubic fields} \label{NE}

Davenport showed that there are only finitely many norm-Euclidean cubic fields with
negative discriminant~\cite{davenport.euclidean}.
Heilbronn showed that there are only finitely many norm-Euclidean cyclic cubic fields with positive
discriminant~\cite{heilbronn.cubic},
and the first author completely determined these fields under
the GRH~\cite{mcgown.cyclic.cubic1, mcgown.cyclic.cubic2}.
This leaves open the case of non-cyclic totally real cubic fields.
In fact, Heilbronn says that he would ``be surprised to learn that the analogue of [the finiteness theorem] is true in this case''.
Lemmermeyer carried out computations (up to discriminant $1.3\cdot 10^4$)
in this setting (see \cite{lemmermeyer})
and observed that the percentage of norm-Euclidean fields was decreasing, and consequently he stated that
``it is tempting to conjecture that the norm-Euclidean cubic fields have density 0.''

This leads to the following problem:  Give an upper bound
on the proportion of totally real cubic fields that are norm-Euclidean.  To our knowledge,
no one has given a nontrivial upper bound in this setting, i.e., a bound less than $100\%$.
The first thing one might try to do is to use genus theory;
in light of Corollary~\ref{C:1}, one knows that less than $96.24\%$ of totally real cubic fields are norm-Euclidean.
The question then is whether one can improve on the upper bound coming from genus theory.
Theorem~\ref{T:euclidean.genus} accomplishes this, albeit very modestly.
In order to give our results in more detail, we must first state Heilbronn's criterion
in our situation.

Let $K$ be a totally real cubic field, and adopt the previous notation for $\Delta,d,f$.
Denote by $F$ the product of all the totally ramified primes in $K$.  Notice that $F=f$ or $F=3f$
depending upon whether $3$ is totally ramified.
The following is the natural adaptation to our setting of a result of Heilbronn
on cyclic cubic fields, which has its roots in a theorem of Erd\"os--Ko~\cite{erdos.ko};
it is also a special case of a more general theorem due to Egami~\cite{egami} who attributed it to Lenstra.
\begin{lemma}[Heilbronn's criterion]\label{L:Heilbronn}
If we can write $F=a+b$ with $a,b\in\Z^+$ where
$a,b$ are not norms and $a$ is a cubic residue modulo $F$,
then $K$ is not norm-Euclidean.
\end{lemma}

One amusing observation is that if $p \not\equiv 1 \pmod{3}$ then every number is a cubic residue modulo $p$,
so Heilbronn's criterion is more easily verified in the genus number one setting (where all $p$ dividing $F$ have this property).

\begin{lemma}
Suppose $K$ has genus number one.
If we can write $F=a+b$ with $a,b\in\Z^+$ where
$a,b$ are not norms, then $K$ is not norm-Euclidean.
\end{lemma}

Let $H(X)$ denote the number of genus number one cubic fields with $0<\Delta\leq X$ to which Heilbronn's criterion applies,
and let $H(F;X)$ denote the number of such fields with fixed $F$.  We have
\[
  H(X)=\sum_{F\leq X^{1/2}}H(F;X)
\]

\begin{proposition}\label{P:HFX}
We have
\[
  H(F;X)=\frac{b_F}{12\zeta(2)}\prod_{p|F}\frac{1}{p(p+1)}X + O(e^{4F/17}X^{16/17+\eps})
\]
for some explicitly computable $b_F\in\Q\cap[0,1]$.
%where
%\[
%C=\frac{29}{324}\prod_{\substack{p\equiv 2\pmod{3}}} \left(1+\frac{1}{p(p+1)}\right)
%\]
\end{proposition}
\begin{center}
\begin{tabular}{|c|c|c|c|c|c|c|c|c|c|c|}%c|c|c|c|c|}
\hline
$F$ & $1$ & $2$ & $3$ & $5$ & $6$ &  $10$ & $11$ & $15$ & $17$ & $22$\\ % & $23$ & $29$ & $30$ & $33$ & $34$ \\ %19 entries
\hline
&&&&&&&&&&\\[-1.5ex]
$b_F$ & $0$ & $0$ & $0$ & $\displaystyle\frac{1}{18}$& $0$& $\displaystyle\frac{7}{96}$ &
$\displaystyle\frac{55}{288}$ &
$\displaystyle\frac{1574}{15309}$ & $\displaystyle\frac{231205}{653184}$ & $\displaystyle\frac{1292771}{4354560}$ \\[2ex]
\hline
$\approx$ & $0$ & $0$ & $0$ & $0.556$& $0$& $0.0729$ & $0.191$ & $0.103$ & $0.354$ & $0.297$ \\
\hline
\end{tabular}
\end{center}

%(5, 1/18, 0.0555555555555556)
%(6, 0, 0.000000000000000)
%(10, 7/96, 0.0729166666666667)
%(11, 55/288, 0.190972222222222)
%(15, 1574/15309, 0.102815337383239)
%(17, 231205/653184, 0.353966110621203)
%(22, 1292771/4354560, 0.296877526087596)

Using the previous proposition, we will prove the main result of this section which
immediately implies Theorem~\ref{T:euclidean.genus}.

\begin{theorem}\label{P:HX}
We have $H(X)\sim B X$ with
$5.7\cdot 10^{-4}\leq B\leq 6.1\cdot 10^{-4}$.
Consequently, Heilbronn's criterion applies to strictly between $4/5$ of a percent and $1$ percent of all totally real cubic fields with genus number one.
\end{theorem}

Admittedly, the proportion
in the previous theorem is rather small.
The main interest here is to know that such a density exists and is positive.
We obtain the rather weak corollary that less than
$96\%$ of totally real cubic fields are norm-Euclidean.
However, this theorem does say something about the limitations of these methods;
in particular, one cannot hope to beat $95\%$ by using only  genus theory and totally ramified primes
(via Lemma~\ref{L:Heilbronn}) --- one would need to inject some new ideas.
In principle, one could use the ideas presented here to compute $B$
more accurately, but we have not pursued this (because of the reasons just mentioned).

\begin{remark*}
Even though the total proportion given in Theorem~\ref{P:HX} is rather small,
if we only consider fields where $F$ is large we can give much stronger results.
For example, when $F=167$, we compute $B_F\approx 0.9421$.  In fact, the precise number is:
\[
\hspace{-2ex}
\frac{5707366742127207720711393876905481748779979640006818006913}{6058037125307413601957148346537399067112071383363249766400}
\]
The upshot is the following:
Suppose we know that $167$ is the only totally ramified prime in $K$.
Then $K$ has less than a $5.8\%$ chance of being norm-Euclidean.
%\[
%b_{101}=\frac{11599896882198978763320285159386592301}{13536816136335754699364804692279296000}\approx 0.85691 %4710621084)
%\]
\end{remark*}

Before launching into the proofs, we recast Heilbronn's criterion into a form that
is more convenient for our purposes.
We define the set
\[
  S=\{a\in\Z\mid 0<a<F\,,\;a\notin N_{K/\Q}(\OK)\}
  \, ,
\]
and rewrite the condition in Heilbronn's criterion as
\[
  (\dagger)\;  \exists \, a\in S\text{ such that } F-a\in S
  \,.
\]
Recall that $n\neq 0$ is a norm if and only if $3 | v_p(n)$ for all inert $p | n$.
In light of this, we immediately see that $(\dagger)$ holds iff there exists an $a\in(0,F)$ and a pair
of inert primes $\{p,q\}$ such that $p|a$, $q|F-a$, and $3\ndiv v_p(a)v_q(F-a)$.
If this condition is satisfied, we call the set of two primes $\{p,q\}$ a Heilbronn pair for~$F$.

\begin{example*}
We find all the Heilbronn pairs for $F=11$.
By symmetry, there are $5$ possible choices of $(a,F-a)$ we must consider.
We can reject the values $(1,10)$ and $(3,8)$ because they contain cubes, which leaves three remaining choices for $(a,F-a)$.
The choice $(2,9)$ leads to the H-pair $\{2,3\}$, the choice $(4,7)$ leads to the H-pair $\{2,7\}$, and the choice
$(5,6)$ leads to the H-pairs $\{5,2\}$, $\{5,3\}$.  In summary, the Heilbronn pairs
for $F=11$ are $\{2,3\}$, $\{2,5\}$, $\{2,7\}$, $\{3,5\}$.  This means
Heilbronn's criterion applies to a cubic field $K$ with $F=11$ if and only if
the collection of primes that are inert in $K$ contains at least one of these four $H$-pairs.
\end{example*}

Let $\I$ be a subset of the primes $p$ with  $p<F$.
We say that $\I$ is admissible if it contains both primes in a Heilbronn pair.
Let $H(F,\I;X)$ denote the number of such fields
where $\I$ is exactly the inert primes less than $F$.
In light of discussion above, we have
\begin{equation}\label{E:finite}
  H(F;X)=\sum_{\I\text{ admissible}}H(F,\I;X)
  \,.
\end{equation}
\begin{proposition}\label{P:HFIX}
Suppose $3$ is not totally ramified in $K$.  Then
\[
H(f,\I;X)=
\frac{1}{12\zeta(2)}\prod_{p|f}\frac{1}{p(p+1)}\prod_{\substack{p<f\\p\not |f}}\frac{a_p}{1+p^{-1}}X
+O(e^{4f/17}X^{16/17+\eps})
\]
where
\[
  a_p=
  \begin{cases}
  1/3 & p\in \I\\
  2/3+1/p & p\not\in \I\\
  \end{cases}
\]
\end{proposition}

\begin{proof}
%We wish to count genus number one fields that satisfy Heilbronn's criterion.
Our hypothesis gives $\Delta=df^2$.
The constant in the main term is:
\begin{align*}
&
\frac{1}{12\zeta(3)}\prod_{p|f}\frac{1/p^2}{1+p^{-1}+p^{-2}}\prod_{\substack{p<f\\p\not |f}}\frac{a_p}{1+p^{-1}+p^{-2}}
\sum_{\substack{r> f\\(r,\prod_{p\leq f}p)=1}}\mu(r)\prod_{p|r}\frac{1}{p^2+p+1}
\\
&
=
\frac{1}{12\zeta(3)}\prod_{p|f}\frac{1}{p^2+p+1}\prod_{\substack{p<f\\p\not |f}}\frac{a_p}{1+p^{-1}+p^{-2}}
\prod_{p>f}\left(1-\frac{1}{p^2+p+1}\right)
\\
&
=
\frac{1}{12\zeta(2)}\prod_{p|f}\frac{1}{p^2+p+1}\prod_{\substack{p<f\\p\not |f}}\frac{a_p}{1+p^{-1}+p^{-2}}
\prod_{p\leq f}\left(1-\frac{1}{p^2+p+1}\right)^{-1}
\\
&
=
\frac{1}{12\zeta(2)}\prod_{p|f}\frac{1}{p(p+1)}\prod_{\substack{p<f\\p\not |f}}\frac{a_p}{1+p^{-1}}
%\prod_{p\leq f}\left(1-\frac{1}{p^2+p+1}\right)^{-1}
\end{align*}
%In the above $a_p=1/3$ or $a_p=2/3+1/p$ depending upon whether $a_p$ is inert or not.
We follow the same procedure as in Section~\ref{genus1}.
However, this time there are many more local conditions being imposed.
When the smoke clears, the error term is equal to
\[
  O(f^{-2}X^{5/6})
  +
  O(Y^{25/9}f^{8/9}e^{8f/9}X^{7/9+\eps})
  +
  O(Y^{-1+\eps}f^{-2+\eps}X)
  \,.
\]
Setting $Y=X^{1/17}f^{-13/17}e^{-4f/17}$ yields the error term
\[
  O(f^{-21/17+\eps}e^{4f/17}X^{16/17+\eps})
  \,.
\]
\end{proof}

If $3$ is totally ramified in $K$, then the previous proposition still holds, but with
$f$ replaced by $F=3f$ and the constant in the main term multiplied by a factor
of $8/9$.

\begin{example*}
We return to our example of $F=11$.
We saw that the four H-pairs in this situation are $\{2,3\}$, $\{2,5\}$, $\{2,7\}$, $\{3,5\}$.
Consequently, there are $9$ admissible sets $\I$; namely:
$\{2,3\}$, $\{2,3,5\}$, $\{2,3,7\}$, $\{2,3,5,7\}$, $\{2,5\}$, $\{2,5,7\}$, $\{2,7\}$, $\{3,5\}$, $\{3,5,7\}$.
Consider for the moment the choice $\I=\{2,3,5\}$.  In this situation, the
extra factor is
\[
   \frac{1/3}{1+2^{-1}}\cdot\frac{1/3}{1+3^{-1}}\cdot\frac{1/3}{1+5^{-1}}\cdot\frac{2/3+1/7}{1+7^{-1}}=\frac{85}{7776}
\]
and hence
Proposition~\ref{P:HFIX}  yields
\[
  H(F,\I;X)\sim
  \frac{85}{7776}
  \cdot
  \frac{1}{12\zeta(2)}\prod_{p|f}\frac{1}{p(p+1)}X
\]
For each admissible $\I$ we get an additional rational factor; summing over all
admissible $\I$ yields:
\[
  H(F;X)\sim
    \frac{55}{288}
  \cdot
  \frac{1}{12\zeta(2)}\prod_{p|f}\frac{1}{p(p+1)}X
\]
\end{example*}

\begin{proof}[Proof of Proposition~\ref{P:HFX}]
This follows immediately from Proposition~\ref{P:HFIX}
since the sum appearing in (\ref{E:finite}) is finite.
The calculation of the $b_F$ is along the lines of the previous example; namely, when $F$ is not divisible by $3$,
\[
  b_F=
%  \frac{1}{12\zeta(2)}\prod_{p|F}\frac{1}{p(p+1)}
  \sum_{\I\text{ admissible}}
  \prod_{\substack{p<F\\p\not |F}}\frac{a_p}{1+p^{-1}}
  \, ,
\]
 and $b_F$ is equal to the same
expression times $8/9$ when $F$ is divisible by $3$.
\end{proof}

\begin{proof}[Proof of Theorem~\ref{P:HX}]
Using Proposition~\ref{P:HFX} for $F<Y$ we obtain:
\begin{align*}
  H(X)
  &=
  \sum_{F<Y} H(F;X)
  +
  \sum_{F>Y}
    O(F^{-2+\eps}X)
    \\
    &=  
    \frac{1}{12\zeta(2)}X\sum_{F} b_F\prod_{p|f}\frac{1}{p(p+1)}
    +
      \sum_{F<Y} O(e^{4F/17}X^{16/17+\eps})
      +
    \sum_{F>Y}
    O(F^{-2+\eps}X)
  \\
  &=
      \frac{1}{12\zeta(2)}X\sum_{F} b_F\prod_{p|f}\frac{1}{p(p+1)}
    +
      O(e^{4Y/17}X^{16/17+\eps})
      +
    O(Y^{-1+\eps}X)
\end{align*}
Choosing $Y$ to be a small power of $\log X$ proves the result with
\[
  B= \frac{1}{12\zeta(2)}\sum_{F} b_F\prod_{p|f}\frac{1}{p(p+1)}
  \,.
\]
\end{proof}

\section*{Acknowledgements}
The authors wish to thank Manjul Bhargava, Henrik Lenstra, Carl Pomerance, Frank Thorne, and Melanie Matchett Wood  for productive input.

\nocite{belabas1, belabas2, ishida}
\bibliographystyle{alpha}
\bibliography{counting}

\vspace{1ex}
{\footnotesize
\noindent
Kevin J. McGown\\
Department of Mathematics and Statistics\\
California State University, Chico\\
Holt 181, 400 West First Street, Chico, California, 95929\\[1ex]

\noindent
Amanda Tucker\\
Department of Mathematics\\
University of Rochester\\
915 Hylan Building, P.O. Box 270138, Rochester, NY 14627
}

\end{document}